\documentclass{amsart}
\usepackage{amssymb}
\usepackage{amsmath}
\usepackage{amsfonts}

\newtheorem{theorem}{Theorem}
\theoremstyle{plain}

\newtheorem{corollary}{Corollary}

\newtheorem{definition}{Definition}

\newtheorem{remark}{Remark}

\numberwithin{equation}{section}
\input{tcilatex}

\begin{document}
\title[Thinness and Fine Topology with Relative Capacity]{Some Properties of
Thinness and Fine Topology with Relative Capacity}
\author{Cihan UNAL}
\address{Sinop University\\
Faculty of Arts and Sciences\\
Department of Mathematics}
\email{cihanunal88@gmail.com}
\urladdr{}
\thanks{}
\author{Ismail AYDIN}
\address{Sinop University\\
Faculty of Arts and Sciences\\
Department of Mathematics}
\email{iaydin@sinop.edu.tr}
\urladdr{}
\thanks{}
\subjclass[2000]{Primary 31C40, 46E35; Secondary 32U20, 43A15}
\keywords{Fine topology, Thinness, Relative capacity, Weighted variable
exponent Sobolev spaces}
\dedicatory{}
\thanks{}

\begin{abstract}
In this paper, we introduce a thinness in sense to a type of relative
capacity for weighted variable exponent Sobolev space. Moreover, we reveal
some properties of this thinness and consider the relationship with finely
open and finely closed sets. We discuss fine topology and compare this
topology with Euclidean one. Finally, we give some information about
importance of the fine topology in the potential theory.
\end{abstract}

\maketitle

\section{Introduction}

The history of potential theory begins in 17th century. Its development can
be traced to such greats as Newton, Euler, Laplace, Lagrange, Fourier,
Green, Gauss, Poisson, Dirichlet, Riemann, Weierstrass, Poincar\'{e}. We
refer to the book by Kellogg \cite{Kel} for references to some of the old
works.

The Sobolev spaces $W^{k,p}\left( \Omega \right) $ are usually defined for
open sets $\Omega .$ This makes sometimes difficulties to classical method
for nonopen sets. The authors in \cite{Kil2} and \cite{Mal} present
different approach is to investigate Sobolev spaces on finely open sets.
This is just a part of fine potential theory in $%
\mathbb{R}
^{d}$.

Kov\'{a}\v{c}ik and R\'{a}kosn\'{\i}k \cite{K} introduced the variable
exponent Lebesgue space $L^{p\left( .\right) }(%
\mathbb{R}
^{d})$ and the Sobolev space $W^{k,p(.)}\left( 
\mathbb{R}
^{d}\right) $. They present some basic properties of the variable exponent
Lebesgue space $L^{p\left( .\right) }(%
\mathbb{R}
^{d})$ and the Sobolev space $W^{k,p(.)}\left( 
\mathbb{R}
^{d}\right) $ such as reflexivity and H\"{o}lder inequalities were obtained.
For a historical journey, we refer \cite{Dien2}, \cite{Fan}, \cite{K}, \cite%
{Mus} and \cite{Sam}.

The variational capacity has been used extensively in nonlinear potential
theory on $%
\mathbb{R}
^{d}$. Let \ $\Omega \subset 
\mathbb{R}
^{d}$ is open and $K\subset \Omega $ is compact. Then the relative
variational $p$-capacity is defined by%
\begin{equation*}
cap_{p}\left( K,\Omega \right) =\inf_{f}\dint\limits_{\Omega }\left\vert
\bigtriangledown f\left( x\right) \right\vert ^{p}dx\text{,}
\end{equation*}%
where the infimum is taken over smooth and zero boundary valued functions $f$
in $\Omega $ such that $f\geq 1$ in $K.$ The set of admissible functions $f$
can be replaced by the continuous first order Sobolev functions with $f\geq
1 $ in $K.$ The $p$-capacity is a Choquet capacity relative to $\Omega .$
For more details and historical background, see \cite{Hei}. Also, Harjulehto
et al. \cite{Har1} defined a relative capacity with variable exponent. They
studied properties of the capacity and compare it with the Sobolev capacity.
In \cite{Unal}, the authors expanded this relative capacity to weighted
variable exponent. Moreover, they investigate properties of this capacity
and give some relationship between defined capacity in \cite{Har1} and
Sobolev capacity. Besides to these studies, the Riesz capacity which is an
another representative for capacity theory has been considered by \cite%
{Unal3}.

In \cite{Ac} and \cite{Cos}, the authors have explored some properties of
the $p\left( .\right) $-Dirichlet energy integral%
\begin{equation*}
\dint\limits_{\Omega }\left\vert \nabla f\left( x\right) \right\vert
^{p\left( x\right) }dx
\end{equation*}%
over a bounded domain $\Omega \subset 
\mathbb{R}
^{d}.$ They have discussed the existence and regularity of energy integral
minimizers. As an alternative method the minimizers in one dimensional case
have been studied by the authors in \cite{Har3}. Moreover, Harjulehto et al. 
\cite{Har4} considered the Dirichlet energy integral, with boundary values
given in the Sobolev sense, has a minimizer provided the variable exponent
satisfies a certain jump condition.

The fine topology was introduced by Cartan \cite{Car} in 1946. Classical
fine topology has found many applications such as its connections to the
theory of analytic functions and probability. For classical treatment we can
refer \cite{Bjorn2008}, \cite{Cons}, \cite{Doob}, \cite{Fug} and \cite{Helm}%
. Also, Meyers \cite{Mey} first generalized the fine topology to nonlinear
theories. For the historical background and an excellent scientific survey
we refer \cite{Hei} and references therein.

In this study, we present $\left( p\left( .\right) ,\vartheta \right) $-thin
sets in sense to $\left( p\left( .\right) ,\vartheta \right) $-relative
capacity and consider the basic and advanced properties. We discuss some
results about $\left( p\left( .\right) ,\vartheta \right) $-relative
capacity in $\left( p\left( .\right) ,\vartheta \right) $-thin sets.
Moreover, we generalize several properties of fine topology and find new
results by Wiener type integral.

\section{Notation and Preliminaries}

In this paper, we will work on $%
\mathbb{R}
^{d}$ with Lebesgue measure $dx$. The measure $\mu $ is doubling if there is
a fixed constant $c_{d}\geq 1,$ called the doubling constant of $\mu $ such
that%
\begin{equation*}
\mu \left( B\left( x_{0},2r\right) \right) \leq c_{d}\mu \left( B\left(
x_{0},r\right) \right)
\end{equation*}%
for every ball $B\left( x_{0},r\right) $ in $%
\mathbb{R}
^{d}.$ Also, the elements of the space $C_{0}^{\infty }\left( 
\mathbb{R}
^{d}\right) $ are the infinitely differentiable functions with compact
support. We denote the family of all measurable functions $p\left( .\right) :%
\mathbb{R}
^{d}\rightarrow \lbrack 1,\infty )$ (called the variable exponent on $%
\mathbb{R}
^{d}$) by the symbol $\mathcal{P}\left( 
\mathbb{R}
^{d}\right) $. In this paper, the function $p(.)$ always denotes a variable
exponent. For $p\left( .\right) \in \mathcal{P}\left( 
\mathbb{R}
^{d}\right) ,$ put 
\begin{equation*}
p^{-}=\underset{x\in 
\mathbb{R}
^{d}}{\text{ess inf}}p(x)\text{, \ \ \ \ \ \ }p^{+}=\underset{x\in 
\mathbb{R}
^{d}}{\text{ess sup}}p(x)\text{.}
\end{equation*}

A measurable and locally integrable function $\vartheta :%
\mathbb{R}
^{d}\rightarrow \left( 0,\infty \right) $ is called a weight function. The
weighted modular is defined by 
\begin{equation*}
\varrho _{p(.),\vartheta }\left( f\right) =\dint\limits_{%
\mathbb{R}
^{d}}\left\vert f(x)\right\vert ^{p(x)}\vartheta \left( x\right) dx\text{.}
\end{equation*}%
The weighted variable exponent Lebesgue spaces $L_{\vartheta }^{p(.)}\left( 
\mathbb{R}
^{d}\right) $ consist of all measurable functions $f$ on $%
\mathbb{R}
^{d}$ endowed with the Luxemburg norm%
\begin{equation*}
\left\Vert f\right\Vert _{p\left( .\right) ,\vartheta }=\inf \left\{ \lambda
>0:\dint\limits_{%
\mathbb{R}
^{d}}\left\vert \frac{f(x)}{\lambda }\right\vert ^{p(x)}\vartheta \left(
x\right) dx\leq 1\right\} .
\end{equation*}%
When $\vartheta \left( x\right) =1,$ the space $L_{\vartheta }^{p(.)}\left( 
\mathbb{R}
^{d}\right) $ is the variable exponent Lebesgue space. The space $%
L_{\vartheta }^{p(.)}\left( 
\mathbb{R}
^{d}\right) $ is a Banach space with respect to $\left\Vert .\right\Vert
_{p(.),\vartheta }.$ Also, some basic properties of this space were
investigated in \cite{A1}, \cite{A2}, \cite{Kok}.

We set the weighted variable exponent Sobolev spaces $W_{\vartheta
}^{k,p\left( .\right) }\left( 
\mathbb{R}
^{d}\right) $ by

\begin{equation*}
W_{\vartheta }^{k,p(.)}(%
\mathbb{R}
^{d})=\left\{ f\in L_{\vartheta }^{p\left( .\right) }\left( 
\mathbb{R}
^{d}\right) :D^{\alpha }f\in L_{\vartheta }^{p(.)}(%
\mathbb{R}
^{d}),0\leq \left\vert \alpha \right\vert \leq k\right\}
\end{equation*}%
with the norm

\begin{equation*}
\left\Vert f\right\Vert _{k,p\left( .\right) ,\vartheta }=\sum_{0\leq
\left\vert \alpha \right\vert \leq k}\left\Vert D^{\alpha }f\right\Vert
_{p\left( .\right) ,\vartheta }
\end{equation*}%
where $\alpha \in 
\mathbb{N}
_{0}^{d}$ is a multiindex, $\left\vert \alpha \right\vert =\alpha
_{1}+\alpha _{2}+...+\alpha _{d},$ and $D^{\alpha }=\frac{\partial
^{\left\vert \alpha \right\vert }}{\partial _{x_{1}}^{\alpha _{1}}\partial
_{x_{2}}^{\alpha _{2}}...\partial _{x_{d}}^{\alpha _{d}}}.$ It is already
known that $W_{\vartheta }^{k,p(.)}\left( 
\mathbb{R}
^{d}\right) $ is a reflexive Banach space.

Now, let $1<p^{-}\leq p\left( .\right) \leq p^{+}<\infty $, $k\in 
\mathbb{N}
$ and $\vartheta ^{-\frac{1}{p\left( .\right) -1}}\in L_{loc}^{1}\left( 
\mathbb{R}
^{d}\right) .$ Thus, the embedding $L_{\vartheta }^{p\left( .\right) }\left( 
\mathbb{R}
^{d}\right) \hookrightarrow L_{loc}^{1}\left( 
\mathbb{R}
^{d}\right) $ holds and then the weighted variable exponent Sobolev spaces $%
W_{\vartheta }^{k,p\left( .\right) }\left( 
\mathbb{R}
^{d}\right) $ is well-defined by [\cite{A2}, Proposition 2.1].

In particular, the space $W_{\vartheta }^{1,p\left( .\right) }\left( 
\mathbb{R}
^{d}\right) $ is defined by%
\begin{equation*}
W_{\vartheta }^{1,p(.)}\left( 
\mathbb{R}
^{d}\right) =\left\{ f\in L_{\vartheta }^{p\left( .\right) }\left( 
\mathbb{R}
^{d}\right) :\left\vert \nabla f\right\vert \in L_{\vartheta }^{p(.)}(%
\mathbb{R}
^{d})\right\} .
\end{equation*}%
The function $\rho _{1,p\left( .\right) ,\vartheta }:W_{\vartheta }^{1,p(.)}(%
\mathbb{R}
^{d})\longrightarrow \left[ 0,\infty \right) $ is shown as $\rho _{1,p\left(
.\right) ,\vartheta }\left( f\right) =\rho _{p\left( .\right) ,\vartheta
}\left( f\right) +\rho _{p\left( .\right) ,\vartheta }\left( \left\vert
\nabla f\right\vert \right) .$ Also, the norm $\left\Vert f\right\Vert
_{1,p\left( .\right) ,\vartheta }=\left\Vert f\right\Vert _{p\left( .\right)
,\vartheta }+\left\Vert \nabla f\right\Vert _{p\left( .\right) ,\vartheta }$
makes the space $W_{\vartheta }^{1,p\left( .\right) }\left( 
\mathbb{R}
^{d}\right) $ a Banach space. The local weighted variable exponent Sobolev
space $W_{\vartheta ,loc}^{1,p\left( .\right) }\left( 
\mathbb{R}
^{d}\right) $ is defined in the classical way. More information on the
classic theory of variable exponent spaces can be found in \cite{Dien4},\cite%
{K}.

Let $\Omega \subset 
\mathbb{R}
^{d}$ is bounded and $\vartheta $ is a weight function$.$ It is known that a
function $f\in C_{0}^{\infty }\left( \Omega \right) $ satisfy Poincar\'{e}
inequality in $L_{\vartheta }^{1}(\Omega )$ if and only if the inequality%
\begin{equation*}
\dint\limits_{\Omega }\left\vert f(x)\right\vert \vartheta \left( x\right)
dx\leq c\left( diam\text{ }\Omega \right) \dint\limits_{\Omega }\left\vert
\nabla f(x)\right\vert \vartheta \left( x\right) dx
\end{equation*}%
holds \cite{Hei}.

Unal and Ayd\i n \cite{Unal} defined an alternative capacity -called
relative $\left( p\left( .\right) ,\vartheta \right) $-capacity-for Sobolev
capacity in sense to \cite{Har1}. For this, they recall that%
\begin{equation*}
C_{0}(\Omega )=\left\{ f:\Omega \longrightarrow 
\mathbb{R}
:f\text{ is continuous and supp}f\subset \Omega \text{ is compact}\right\} ,
\end{equation*}%
where supp$f$ is the support of $f$. Suppose that $K$ is a compact subset of 
$\Omega .$ Also, they denote%
\begin{equation*}
R_{p\left( .\right) ,\vartheta }\left( K,\Omega \right) =\left\{ f\in
W_{\vartheta }^{1,p(.)}\left( \Omega \right) \cap C_{0}\left( \Omega \right)
:f>1\text{ on }K\text{ and }f\geq 0\right\}
\end{equation*}%
and define%
\begin{equation*}
cap_{p\left( .\right) ,\vartheta }^{\ast }\left( K,\Omega \right)
=\inf_{f\in R_{p\left( .\right) ,\vartheta }\left( K,\Omega \right)
}\dint\limits_{\Omega }\left\vert \bigtriangledown f\left( x\right)
\right\vert ^{p\left( x\right) }\vartheta \left( x\right) dx.
\end{equation*}

In addition, if $U\subset \Omega $ is open, then%
\begin{equation*}
cap_{p\left( .\right) ,\vartheta }\left( U,\Omega \right) =\sup_{\substack{ %
K\subset U  \\ compact}}cap_{p\left( .\right) ,\vartheta }^{\ast }\left(
K,\Omega \right) ,
\end{equation*}%
and also for an arbitrary set $E\subset \Omega $ we define%
\begin{equation*}
cap_{p\left( .\right) ,\vartheta }\left( E,\Omega \right) =\inf_{\substack{ %
E\subset U\subset \Omega  \\ U\text{ }open}}cap_{p\left( .\right) ,\vartheta
}\left( U,\Omega \right) .
\end{equation*}%
They call $cap_{p\left( .\right) ,\vartheta }\left( E,\Omega \right) $ the
variational $\left( p\left( .\right) ,\vartheta \right) $-capacity of $E$
relative to $\Omega $, briefly the relative $\left( p\left( .\right)
,\vartheta \right) $-capacity. Also, the relative $\left( p\left( .\right)
,\vartheta \right) $-capacity has the following properties.

\begin{enumerate}
\item[P1] . $cap_{p\left( .\right) ,\vartheta }\left( \emptyset ,\Omega
\right) =0.$

\item[P2] . If $E_{1}\subset E_{2}\subset \Omega _{2}\subset \Omega _{1},$
then $cap_{p\left( .\right) ,\vartheta }\left( E_{1},\Omega _{1}\right) \leq
cap_{p\left( .\right) ,\vartheta }\left( E_{2},\Omega _{2}\right) .$

\item[P3] . If $E$ is a subset of $\Omega ,$ then%
\begin{equation*}
cap_{p\left( .\right) ,\vartheta }\left( E,\Omega \right) =\inf_{\substack{ %
E\subset U\subset \Omega  \\ U\text{ }open}}cap_{p\left( .\right) ,\vartheta
}\left( U,\Omega \right) .
\end{equation*}

\item[P4] . If $K_{1}$ and $K_{2}$ are compact subsets of $\Omega ,$ then%
\begin{eqnarray*}
cap_{p\left( .\right) ,\vartheta }\left( K_{1}\cup K_{2},\Omega \right)
+cap_{p\left( .\right) ,\vartheta }\left( K_{1}\cap K_{2},\Omega \right)
&\leq &cap_{p\left( .\right) ,\vartheta }\left( K_{1},\Omega \right) \\
&&+cap_{p\left( .\right) ,\vartheta }\left( K_{2},\Omega \right) .
\end{eqnarray*}

\item[P5] . Let $K_{n}$ is a decreasing sequence of compact subsets of $%
\Omega $ for $n\in 
\mathbb{N}
.$ Then%
\begin{equation*}
\lim_{n\longrightarrow \infty }cap_{p\left( .\right) ,\vartheta }\left(
K_{n},\Omega \right) =cap_{p\left( .\right) ,\vartheta }\left(
\tbigcap\limits_{n=1}^{\infty }K_{n},\Omega \right) .
\end{equation*}

\item[P6] . If $E_{n}$ is an increasing sequence of subsets of $\Omega $ for 
$n\in 
\mathbb{N}
,$ then%
\begin{equation*}
\lim_{n\longrightarrow \infty }cap_{p\left( .\right) ,\vartheta }\left(
E_{n},\Omega \right) =cap_{p\left( .\right) ,\vartheta }\left(
\tbigcup\limits_{n=1}^{\infty }E_{n},\Omega \right) .
\end{equation*}

\item[P7] . If $E_{n}\subset \Omega $ for $n\in 
\mathbb{N}
,$ then%
\begin{equation*}
cap_{p\left( .\right) ,\vartheta }\left( \tbigcup\limits_{n=1}^{\infty
}E_{n},\Omega \right) \leq \tsum\limits_{n=1}^{\infty }cap_{p\left( .\right)
,\vartheta }\left( E_{n},\Omega \right) .
\end{equation*}
\end{enumerate}

\begin{theorem}
\label{fineatif1}\cite{Unal}If $cap_{p\left( .\right) ,\vartheta }\left(
B\left( x_{0},r\right) ,B\left( x_{0},2r\right) \right) \geq 1$ and $\mu
_{\vartheta }$ is a doubling measure, then there exist positive constants $%
C_{1},C_{2}$ such that%
\begin{equation*}
C_{1}\mu _{\vartheta }\left( B\left( x_{0},r\right) \right) \leq
cap_{p\left( .\right) ,\vartheta }\left( B\left( x_{0},r\right) ,B\left(
x_{0},2r\right) \right) \leq C_{2}\mu _{\vartheta }\left( B\left(
x_{0},r\right) \right)
\end{equation*}%
where the constants depend on $r,p^{-},p^{+}$, constants of doubling measure
and Poincar\'{e} inequality.
\end{theorem}

\begin{theorem}
\label{fineatif2}\cite{Unal}If $E\subset B\left( x_{0},r\right) ,$ $%
cap_{p\left( .\right) ,\vartheta }\left( E,B\left( x_{0},4r\right) \right)
\geq 1$ and $0<r\leq s\leq 2r,$ then the inequality%
\begin{equation*}
\frac{1}{C}cap_{p\left( .\right) ,\vartheta }\left( E,B\left(
x_{0},2r\right) \right) \leq cap_{p\left( .\right) ,\vartheta }\left(
E,B\left( x_{0},2s\right) \right) \leq cap_{p\left( .\right) ,\vartheta
}\left( E,B\left( x_{0},2r\right) \right)
\end{equation*}%
holds where the constants depend on $r,p^{-},p^{+}$, constants of doubling
measure and Poincar\'{e} inequality.
\end{theorem}

The proofs can be found in \cite{Unal}.

We say that a property holds $\left( p(.),\vartheta \right) $%
-quasieverywhere if it satisfies except in a set of capacity zero. Recall
also a \ function $f$ is $\left( p(.),\vartheta \right) $-quasicontinuous in 
$%
\mathbb{R}
^{d}$ if for each $\varepsilon >0$ there exists a set $A$ with the capacity
of $A$ is less than $\varepsilon $ such that $f$ restricted to $%
\mathbb{R}
^{d}-A$ is continuous. If the capacity is an outer capacity, we can suppose
that $A$ is open. More detail can be found in \cite{A2}.

Let $\Omega \subset 
\mathbb{R}
^{d}$ be an open set. The space $W_{0,\vartheta }^{1,p(.)}\left( \Omega
\right) $ is denoted as the set of all measurable functions $f$ if there
exists a $\left( p(.),\vartheta \right) $-q.c. function $f^{\ast }\in
W_{\vartheta }^{1,p(.)}\left( 
\mathbb{R}
^{d}\right) $ such that $f=f^{\ast }$ a.e. in $\Omega $ and $f^{\ast }=0$ $%
\left( p(.),\vartheta \right) $-q.e. in $%
\mathbb{R}
^{d}-\Omega $. In other words, $f\in W_{0,\vartheta }^{1,p(.)}\left( \Omega
\right) ,$ if there exist a $\left( p(.),\vartheta \right) $-q.c. function $%
f^{\ast }\in W_{\vartheta }^{1,p(.)}\left( 
\mathbb{R}
^{d}\right) $ such that the trace of $f^{\ast }$ vanishes. More detail about
the space can be seen by \cite{HarLat}, \cite{Hei}, \cite{Unal2}.

Moreover, $A\Subset B$ means that $\overline{A}$ is a compact subset of $B.$
Throughout this paper, we assume that $1<p^{-}\leq p\left( .\right) \leq
p^{+}<\infty $ and $\vartheta ^{-\frac{1}{p\left( .\right) -1}}\in
L_{loc}^{1}\left( 
\mathbb{R}
^{d}\right) .$ Also, we will denote%
\begin{equation*}
\mu _{\vartheta }\left( \Omega \right) =\dint\limits_{\Omega }\vartheta
\left( x\right) dx.
\end{equation*}

\section{The $\left( p\left( .\right) ,\protect\vartheta \right) $-Thinness
and Fine Topology}

Now, we present $\left( p\left( .\right) ,\vartheta \right) $-thinness and
consider some properties of this thinness before considering the fine
topology.

\begin{definition}
A set $E\subset 
\mathbb{R}
^{d}$ is $\left( p\left( .\right) ,\vartheta \right) $-thin at $x\in 
\mathbb{R}
^{d}$ if%
\begin{equation}
\dint\limits_{0}^{1}\left( \frac{cap_{p\left( .\right) ,\vartheta }\left(
E\cap B\left( x,r\right) ,B\left( x,2r\right) \right) }{cap_{p\left(
.\right) ,\vartheta }\left( B\left( x,r\right) ,B\left( x,2r\right) \right) }%
\right) ^{\frac{1}{p\left( x\right) -1}}\frac{dr}{r}<\infty .  \label{thin}
\end{equation}%
Also, we say that $E$ is $\left( p\left( .\right) ,\vartheta \right) $-thick
at $x\in 
\mathbb{R}
^{d}$ if $E$ is not $\left( p\left( .\right) ,\vartheta \right) $-thin at $%
x\in 
\mathbb{R}
^{d}.$
\end{definition}

In the definition of $\left( p\left( .\right) ,\vartheta \right) $-thinness
we make a convention that the integral is 1 if $cap_{p\left( .\right)
,\vartheta }\left( B\left( x,r\right) ,B\left( x,2r\right) \right) =0$.
Also, the integral in (\ref{thin}) is usually called the Wiener type
integral, briefly Wiener integral, as%
\begin{equation*}
W_{p\left( .\right) ,\vartheta }\left( E,x\right)
=\dint\limits_{0}^{1}\left( \frac{cap_{p\left( .\right) ,\vartheta }\left(
E\cap B\left( x,r\right) ,B\left( x,2r\right) \right) }{cap_{p\left(
.\right) ,\vartheta }\left( B\left( x,r\right) ,B\left( x,2r\right) \right) }%
\right) ^{\frac{1}{p\left( x\right) -1}}\frac{dr}{r}.
\end{equation*}%
In addition, we denote the Wiener sum $W_{p\left( .\right) ,\vartheta
}^{sum}\left( E,x\right) $ as%
\begin{equation*}
W_{p\left( .\right) ,\vartheta }^{sum}\left( E,x\right)
=\dsum\limits_{i=0}^{\infty }\left( \frac{cap_{p\left( .\right) ,\vartheta
}\left( E\cap B\left( x,2^{-i}\right) ,B\left( x,2^{1-i}\right) \right) }{%
cap_{p\left( .\right) ,\vartheta }\left( B\left( x,2^{-i}\right) ,B\left(
x,2^{1-i}\right) \right) }\right) ^{\frac{1}{p\left( x\right) -1}}.
\end{equation*}%
Now we give a relationship between these two notions. The proof can be found
in \cite{Unal}.

\begin{theorem}
\label{finedenklik}Assume that the hypotheses of Theorem \ref{fineatif1} and
Theorem \ref{fineatif2} are hold. Then there exist positive constants $%
C_{1},C_{2}$ such that%
\begin{equation*}
C_{1}W_{p\left( .\right) ,\vartheta }\left( E,x\right) \leq W_{p\left(
.\right) ,\vartheta }^{sum}\left( E,x\right) \leq C_{2}W_{p\left( .\right)
,\vartheta }\left( E,x\right)
\end{equation*}%
for every $E\subset 
\mathbb{R}
^{d}$ and $x_{0}\notin E.$ In particular, $W_{p\left( .\right) ,\vartheta
}\left( E,x_{0}\right) $ is finite if and only if $W_{p\left( .\right)
,\vartheta }^{sum}\left( E,x_{0}\right) $ is finite.
\end{theorem}

The previous theorem tell us that the notions $W_{p\left( .\right)
,\vartheta }$ and $W_{p\left( .\right) ,\vartheta }^{sum}$ are equivalent
under some conditions. In some cases, the Wiener sum $W_{p\left( .\right)
,\vartheta }^{sum}$ is more practical than the Wiener integral $W_{p\left(
.\right) ,\vartheta }$.

\begin{definition}
A set $U\subset 
\mathbb{R}
^{d}$ is called $\left( p\left( .\right) ,\vartheta \right) $-finely open if 
$%
\mathbb{R}
^{d}-U$ is $\left( p\left( .\right) ,\vartheta \right) $-thin at $x\in U.$
Equivalently, a set is $\left( p\left( .\right) ,\vartheta \right) $-finely
closed if it includes all points where it is not $\left( p\left( .\right)
,\vartheta \right) $-thin. Moreover, the fine interior of $A,$ briefly
fine-int$A$, is the largest $\left( p\left( .\right) ,\vartheta \right) $%
-finely open set contained in $A.$ In a similar way, the fine closure of $F,$
briefly fine-clo$F,$ is the smallest $\left( p\left( .\right) ,\vartheta
\right) $-finely closed set containing $F.$
\end{definition}

\begin{theorem}
The $\left( p\left( .\right) ,\vartheta \right) $-fine topology on $%
\mathbb{R}
^{d}$ is generated by $\left( p\left( .\right) ,\vartheta \right) $-finely
open sets.
\end{theorem}

\begin{proof}
Firstly, we denote%
\begin{eqnarray*}
\tau _{F} &=&\left\{ E\subset 
\mathbb{R}
^{d}:\dint\limits_{0}^{1}\left( \frac{cap_{p\left( .\right) ,\vartheta
}\left( \left( 
\mathbb{R}
^{d}-E\right) \cap B\left( x,r\right) ,B\left( x,2r\right) \right) }{%
cap_{p\left( .\right) ,\vartheta }\left( B\left( x,r\right) ,B\left(
x,2r\right) \right) }\right) ^{\frac{1}{p\left( x\right) -1}}\frac{dr}{r}%
<\infty \right\} \cup \emptyset \\
&=&\left\{ E\subset 
\mathbb{R}
^{d}:\dint\limits_{0}^{1}\left( \frac{cap_{p\left( .\right) ,\vartheta
}\left( B\left( x,r\right) -E,B\left( x,2r\right) \right) }{cap_{p\left(
.\right) ,\vartheta }\left( B\left( x,r\right) ,B\left( x,2r\right) \right) }%
\right) ^{\frac{1}{p\left( x\right) -1}}\frac{dr}{r}<\infty \right\} \cup
\emptyset .
\end{eqnarray*}%
It is obvious that $\emptyset \in \tau _{F}.$ Since $cap_{p\left( .\right)
,\vartheta }\left( \emptyset ,B\left( x,2r\right) \right) =0,$ we have%
\begin{eqnarray*}
&&\dint\limits_{0}^{1}\left( \frac{cap_{p\left( .\right) ,\vartheta }\left(
B\left( x,r\right) -%
\mathbb{R}
^{d},B\left( x,2r\right) \right) }{cap_{p\left( .\right) ,\vartheta }\left(
B\left( x,r\right) ,B\left( x,2r\right) \right) }\right) ^{\frac{1}{p\left(
x\right) -1}}\frac{dr}{r} \\
&=&\dint\limits_{0}^{1}\left( \frac{cap_{p\left( .\right) ,\vartheta }\left(
\emptyset ,B\left( x,2r\right) \right) }{cap_{p\left( .\right) ,\vartheta
}\left( B\left( x,r\right) ,B\left( x,2r\right) \right) }\right) ^{\frac{1}{%
p\left( x\right) -1}}\frac{dr}{r}<\infty .
\end{eqnarray*}%
This follows that $%
\mathbb{R}
^{d}\in \tau _{F}.$ Now, we assert that finite intersections of $\left(
p\left( .\right) ,\vartheta \right) $-finely open sets are $\left( p\left(
.\right) ,\vartheta \right) $-finely open. Assume that $x\in
\tbigcap\limits_{i=1}^{n}U_{i}$ where $U_{1},U_{2},...,U_{n}$ are $\left(
p\left( .\right) ,\vartheta \right) $-finely open. Thus, if we consider the
subadditivity of relative $\left( p\left( .\right) ,\vartheta \right) $%
-capacity and the cases of the exponent $p\left( .\right) $ as $1<p\left(
.\right) \leq 2$ and $p\left( .\right) >2,$ then we get%
\begin{eqnarray*}
&&\dint\limits_{0}^{1}\left( \frac{cap_{p\left( .\right) ,\vartheta }\left(
B\left( x,r\right) -\tbigcap\limits_{i=1}^{n}U_{i},B\left( x,2r\right)
\right) }{cap_{p\left( .\right) ,\vartheta }\left( B\left( x,r\right)
,B\left( x,2r\right) \right) }\right) ^{\frac{1}{p\left( x\right) -1}}\frac{%
dr}{r} \\
&\leq &\dint\limits_{0}^{1}\left( \tsum\limits_{i=1}^{n}\frac{cap_{p\left(
.\right) ,\vartheta }\left( B\left( x,r\right) -U_{i},B\left( x,2r\right)
\right) }{cap_{p\left( .\right) ,\vartheta }\left( B\left( x,r\right)
,B\left( x,2r\right) \right) }\right) ^{\frac{1}{p\left( x\right) -1}}\frac{%
dr}{r} \\
&\leq &C\tsum\limits_{i=1}^{n}\dint\limits_{0}^{1}\left( \frac{cap_{p\left(
.\right) ,\vartheta }\left( B\left( x,r\right) -U_{i},B\left( x,2r\right)
\right) }{cap_{p\left( .\right) ,\vartheta }\left( B\left( x,r\right)
,B\left( x,2r\right) \right) }\right) ^{\frac{1}{p\left( x\right) -1}}\frac{%
dr}{r}<\infty
\end{eqnarray*}%
where $C>0$ depends on $n,p^{-},p^{+}.$ Therefore $\tbigcap%
\limits_{i=1}^{n}U_{i}$ is $\left( p\left( .\right) ,\vartheta \right) $%
-finely open. Finally, we need to show that arbitrary unions of $\left(
p\left( .\right) ,\vartheta \right) $-finely open sets are $\left( p\left(
.\right) ,\vartheta \right) $-finely open. Let $x\in \tbigcup\limits_{i\in
I}U_{i}$ where $U_{i},i\in I,$ are $\left( p\left( .\right) ,\vartheta
\right) $-finely open sets, and $I$ is an index set. Thus, for every $i\in
I, $ we have%
\begin{equation}
\dint\limits_{0}^{1}\left( \frac{cap_{p\left( .\right) ,\vartheta }\left(
B\left( x,r\right) -U_{i},B\left( x,2r\right) \right) }{cap_{p\left(
.\right) ,\vartheta }\left( B\left( x,r\right) ,B\left( x,2r\right) \right) }%
\right) ^{\frac{1}{p\left( x\right) -1}}\frac{dr}{r}<\infty .
\label{finelyopen}
\end{equation}%
Moreover, it is clear that $B\left( x,r\right) -\tbigcup\limits_{i\in
I}U_{i}\subset B\left( x,r\right) -U_{j}$ or equivalently $%
\tbigcap\limits_{i\in I}\left( B\left( x,r\right) -U_{i}\right) \subset
B\left( x,r\right) -U_{j}$ for $j\in I.$ If we consider the properties of
relative $\left( p\left( .\right) ,\vartheta \right) $-capacity and (\ref%
{finelyopen}), then we get%
\begin{eqnarray*}
&&\dint\limits_{0}^{1}\left( \frac{cap_{p\left( .\right) ,\vartheta }\left(
B\left( x,r\right) -\tbigcup\limits_{i\in I}U_{i},B\left( x,2r\right)
\right) }{cap_{p\left( .\right) ,\vartheta }\left( B\left( x,r\right)
,B\left( x,2r\right) \right) }\right) ^{\frac{1}{p\left( x\right) -1}}\frac{%
dr}{r} \\
&=&\dint\limits_{0}^{1}\left( \frac{cap_{p\left( .\right) ,\vartheta }\left(
\tbigcap\limits_{i\in I}\left( B\left( x,r\right) -U_{i}\right) ,B\left(
x,2r\right) \right) }{cap_{p\left( .\right) ,\vartheta }\left( B\left(
x,r\right) ,B\left( x,2r\right) \right) }\right) ^{\frac{1}{p\left( x\right)
-1}}\frac{dr}{r} \\
&\leq &\dint\limits_{0}^{1}\left( \frac{cap_{p\left( .\right) ,\vartheta
}\left( B\left( x,r\right) -U_{j},B\left( x,2r\right) \right) }{cap_{p\left(
.\right) ,\vartheta }\left( B\left( x,r\right) ,B\left( x,2r\right) \right) }%
\right) ^{\frac{1}{p\left( x\right) -1}}\frac{dr}{r}<\infty .
\end{eqnarray*}%
Therefore, $%
\mathbb{R}
^{d}-\tbigcup\limits_{i\in I}U_{i}$ is $\left( p\left( .\right) ,\vartheta
\right) $-thin at $x$ and as $x\in \tbigcup\limits_{i\in I}U_{i}$ was
arbitrary, $\tbigcup\limits_{i\in I}U_{i}$ is $\left( p\left( .\right)
,\vartheta \right) $-finely open.
\end{proof}

\begin{corollary}
\label{tersi}Every open set is $\left( p\left( .\right) ,\vartheta \right) $%
-finely open.
\end{corollary}

\begin{proof}
Assume that $A$ is an open set in $%
\mathbb{R}
^{d}.$ For every $x\in A,$ by the definition of openness, there exists a $%
t>0 $ such that $B\left( x,t\right) \subset A.$ It is easy to see that $%
B\left( x,r\right) \subset B\left( x,t\right) \subset A$ for small enough $%
r>0$. This follows that%
\begin{equation*}
\dint\limits_{0}^{1}\left( \frac{cap_{p\left( .\right) ,\vartheta }\left(
B\left( x,r\right) -A,B\left( x,2r\right) \right) }{cap_{p\left( .\right)
,\vartheta }\left( B\left( x,r\right) ,B\left( x,2r\right) \right) }\right)
^{\frac{1}{p\left( x\right) -1}}\frac{dr}{r}<\infty
\end{equation*}%
,that is, $A\subset 
\mathbb{R}
^{d}$ is $\left( p\left( .\right) ,\vartheta \right) $-finely open.
\end{proof}

\begin{remark}
By the similar method in Corollary \ref{tersi}, it can be shown that every
closed set is $\left( p\left( .\right) ,\vartheta \right) $-finely closed
and that finite union of $\left( p\left( .\right) ,\vartheta \right) $%
-finely closed sets is $\left( p\left( .\right) ,\vartheta \right) $-finely
closed again.
\end{remark}

\begin{corollary}
The $\left( p\left( .\right) ,\vartheta \right) $-fine topology generated by
the $\left( p\left( .\right) ,\vartheta \right) $-finely open sets is finer
than Euclidean topology.
\end{corollary}

The opposite claim of Corollary \ref{tersi} is not true in general. To see
this, we give the Lebesgue spine%
\begin{equation*}
E=\left\{ \left( x,t\right) \in 
\mathbb{R}
^{2}\times 
\mathbb{R}
:t>0\text{ and }\left\vert x\right\vert <e^{-\frac{1}{t}}\right\}
\end{equation*}%
as a counter example, see \cite[Example 13.4]{Bjorn2011}.

Now, we consider the more general case in sense to Corollary \ref{tersi}.

\begin{theorem}
Assume that $A\subset 
\mathbb{R}
^{d}$ is an open or $\left( p\left( .\right) ,\vartheta \right) $-finely
open set. Moreover, let, the relative $\left( p\left( .\right) ,\vartheta
\right) $-capacity of $E$ is zero. Then $A-E$ is $\left( p\left( .\right)
,\vartheta \right) $-finely open.
\end{theorem}

\begin{proof}
By the Corollary \ref{tersi}, we can consider that $A\subset 
\mathbb{R}
^{d}$ is an open set. Thus, for all $y\in A,$%
\begin{equation}
\dint\limits_{0}^{1}\left( \frac{cap_{p\left( .\right) ,\vartheta }\left(
B\left( y,r\right) -A,B\left( y,2r\right) \right) }{cap_{p\left( .\right)
,\vartheta }\left( B\left( y,r\right) ,B\left( y,2r\right) \right) }\right)
^{\frac{1}{p\left( y\right) -1}}\frac{dr}{r}<\infty .  \label{finelyopen1}
\end{equation}%
Moreover, if we consider the properties of relative $\left( p\left( .\right)
,\vartheta \right) $-capacity, for all $x\in A-E$ and $r>0$, we have%
\begin{eqnarray}
&&cap_{p\left( .\right) ,\vartheta }\left( B\left( x,r\right) -\left(
A-E\right) ,B\left( x,2r\right) \right)  \notag \\
&=&cap_{p\left( .\right) ,\vartheta }\left( \left( B\left( x,r\right)
-A\right) \cup \left( B\left( x,r\right) \cap E\right) ,B\left( x,2r\right)
\right)  \notag \\
&\leq &cap_{p\left( .\right) ,\vartheta }\left( \left( B\left( x,r\right)
-A\right) ,B\left( x,2r\right) \right)  \notag \\
&&+cap_{p\left( .\right) ,\vartheta }\left( \left( B\left( x,r\right) \cap
E\right) ,B\left( x,2r\right) \right) .  \label{finelyopen2}
\end{eqnarray}%
Using the (\ref{finelyopen1}) and (\ref{finelyopen2}), we get%
\begin{eqnarray*}
&&\dint\limits_{0}^{1}\left( \frac{cap_{p\left( .\right) ,\vartheta }\left(
B\left( x,r\right) -\left( A-E\right) ,B\left( x,2r\right) \right) }{%
cap_{p\left( .\right) ,\vartheta }\left( B\left( x,r\right) ,B\left(
x,2r\right) \right) }\right) ^{\frac{1}{p\left( x\right) -1}}\frac{dr}{r} \\
&\leq &\dint\limits_{0}^{1}\left( \frac{cap_{p\left( .\right) ,\vartheta
}\left( \left( B\left( x,r\right) -A\right) ,B\left( x,2r\right) \right) }{%
cap_{p\left( .\right) ,\vartheta }\left( B\left( x,r\right) ,B\left(
x,2r\right) \right) }\right) ^{\frac{1}{p\left( x\right) -1}}\frac{dr}{r} \\
&&+\dint\limits_{0}^{1}\left( \frac{cap_{p\left( .\right) ,\vartheta }\left(
\left( B\left( x,r\right) \cap E\right) ,B\left( x,2r\right) \right) }{%
cap_{p\left( .\right) ,\vartheta }\left( B\left( x,r\right) ,B\left(
x,2r\right) \right) }\right) ^{\frac{1}{p\left( x\right) -1}}\frac{dr}{r} \\
&<&\infty .
\end{eqnarray*}%
This completes the proof.
\end{proof}

Now, we give that $\left( p\left( .\right) ,\vartheta \right) $-thinness is
a local property.

\begin{theorem}
$A\subset 
\mathbb{R}
^{d}$ is $\left( p\left( .\right) ,\vartheta \right) $-thin at $x\in 
\mathbb{R}
^{d}$ if and only if for any $\delta >0,$ the set $A\cap B\left( x,\delta
\right) $ is $\left( p\left( .\right) ,\vartheta \right) $-thin at $x\in 
\mathbb{R}
^{d}.$
\end{theorem}

\begin{proof}
Let $A\subset 
\mathbb{R}
^{d}$ and $x\in 
\mathbb{R}
^{d}.$ Assume that $A\subset 
\mathbb{R}
^{d}$ is $\left( p\left( .\right) ,\vartheta \right) $-thin at $x\in 
\mathbb{R}
^{d}.$ This follows that%
\begin{equation*}
\dint\limits_{0}^{1}\left( \frac{cap_{p\left( .\right) ,\vartheta }\left(
A\cap B\left( x,r\right) ,B\left( x,2r\right) \right) }{cap_{p\left(
.\right) ,\vartheta }\left( B\left( x,r\right) ,B\left( x,2r\right) \right) }%
\right) ^{\frac{1}{p\left( x\right) -1}}\frac{dr}{r}<\infty .
\end{equation*}%
By the monotonicity of relative $\left( p\left( .\right) ,\vartheta \right) $%
-capacity, we have%
\begin{equation}
\dint\limits_{0}^{1}\left( \frac{cap_{p\left( .\right) ,\vartheta }\left(
A\cap B\left( x,\delta \right) \cap B\left( x,r\right) ,B\left( x,2r\right)
\right) }{cap_{p\left( .\right) ,\vartheta }\left( B\left( x,r\right)
,B\left( x,2r\right) \right) }\right) ^{\frac{1}{p\left( x\right) -1}}\frac{%
dr}{r}<\infty  \label{thick}
\end{equation}%
for any $\delta >0$. This completes the necessary condition part of the
proof. Now, we assume that for any $\delta >0,$ the set $A\cap B\left(
x,\delta \right) $ is $\left( p\left( .\right) ,\vartheta \right) $-thin at $%
x\in 
\mathbb{R}
^{d}.$ Thus (\ref{thick}) is satisfied for all $\delta >0,$ in particular,
for $0<r\leq \delta \leq 1.$ Let $A$ is $\left( p\left( .\right) ,\vartheta
\right) $-thick at $x\in 
\mathbb{R}
^{d}.$ Then we have%
\begin{eqnarray*}
&&\dint\limits_{0}^{1}\left( \frac{cap_{p\left( .\right) ,\vartheta }\left(
A\cap B\left( x,\delta \right) \cap B\left( x,r\right) ,B\left( x,2r\right)
\right) }{cap_{p\left( .\right) ,\vartheta }\left( B\left( x,r\right)
,B\left( x,2r\right) \right) }\right) ^{\frac{1}{p\left( x\right) -1}}\frac{%
dr}{r} \\
&=&\dint\limits_{0}^{1}\left( \frac{cap_{p\left( .\right) ,\vartheta }\left(
A\cap B\left( x,r\right) ,B\left( x,2r\right) \right) }{cap_{p\left(
.\right) ,\vartheta }\left( B\left( x,r\right) ,B\left( x,2r\right) \right) }%
\right) ^{\frac{1}{p\left( x\right) -1}}\frac{dr}{r}=\infty .
\end{eqnarray*}%
This is a contradiction. That is the desired result.
\end{proof}

\begin{theorem}
Let the hypotheses of Theorem \ref{finedenklik} hold. Moreover, assume that
there is a point $x\in A$ such that $%
\mathbb{R}
^{d}-A$ is $\left( p\left( .\right) ,\vartheta \right) $-thin at $x.$ Then
there exist a point $x\in A$ and $s>0$ such that%
\begin{equation*}
cap_{p\left( .\right) ,\vartheta }\left( B\left( x,s\right) -A,B\left(
x,2s\right) \right) <cap_{p\left( .\right) ,\vartheta }\left( B\left(
x,s\right) ,B\left( x,2s\right) \right) .
\end{equation*}
\end{theorem}

\begin{proof}
Since $%
\mathbb{R}
^{d}-A$ is $\left( p\left( .\right) ,\vartheta \right) $-thin at $x\in A$,
by the Theorem \ref{finedenklik}, we have%
\begin{eqnarray*}
W_{p\left( .\right) ,\vartheta }^{sum}\left( 
\mathbb{R}
^{d}-A,x\right) &=&\dsum\limits_{i=0}^{\infty }\left( \frac{cap_{p\left(
.\right) ,\vartheta }\left( \left( 
\mathbb{R}
^{d}-A\right) \cap B\left( x,2^{-i}\right) ,B\left( x,2^{1-i}\right) \right) 
}{cap_{p\left( .\right) ,\vartheta }\left( B\left( x,2^{-i}\right) ,B\left(
x,2^{1-i}\right) \right) }\right) ^{\frac{1}{p\left( x\right) -1}} \\
&=&\dsum\limits_{i=0}^{\infty }\left( \frac{cap_{p\left( .\right) ,\vartheta
}\left( B\left( x,2^{-i}\right) -A,B\left( x,2^{1-i}\right) \right) }{%
cap_{p\left( .\right) ,\vartheta }\left( B\left( x,2^{-i}\right) ,B\left(
x,2^{1-i}\right) \right) }\right) ^{\frac{1}{p\left( x\right) -1}}<\infty .
\end{eqnarray*}%
This follows that%
\begin{equation*}
\liminf_{i\longrightarrow \infty }\left( \frac{cap_{p\left( .\right)
,\vartheta }\left( B\left( x,2^{-i}\right) -A,B\left( x,2^{1-i}\right)
\right) }{cap_{p\left( .\right) ,\vartheta }\left( B\left( x,2^{-i}\right)
,B\left( x,2^{1-i}\right) \right) }\right) ^{\frac{1}{p\left( x\right) -1}%
}=0.
\end{equation*}%
By the definition of limit, we get the desired result.
\end{proof}

\begin{remark}
The proof of the previous theorem can be considered by using Wiener integral 
$W_{p\left( .\right) ,\vartheta }\left( A,x\right) $ with similar method.
Here, there is not necessary the condition that the hypotheses of Theorem %
\ref{finedenklik} are hold.
\end{remark}

\begin{definition}
Let $U$ be a $\left( p\left( .\right) ,\vartheta \right) $-finely open set.
A function $f:U\longrightarrow 
\mathbb{R}
$ is $\left( p\left( .\right) ,\vartheta \right) $-finely continuous at $%
x_{0}\in U$ if $\left\{ x\in U:\left\vert f\left( x\right) -f\left(
x_{0}\right) \right\vert \geq \varepsilon \right\} $ is $\left( p\left(
.\right) ,\vartheta \right) $-thin at $x_{0}$ for each $\varepsilon >0.$
\end{definition}

\begin{remark}
Assume that $U$ is a $\left( p\left( .\right) ,\vartheta \right) $-finely
open set and $f:U\longrightarrow 
\mathbb{R}
$ is $\left( p\left( .\right) ,\vartheta \right) $-finely continuous at $%
x_{0}\in U.$ Then $f$ is continuous function with respect to the $\left(
p\left( .\right) ,\vartheta \right) $-fine topology on $U.$ Indeed, if we
consider the definition of finely continuous, then the set $\left\{ x\in
U:\left\vert f\left( x\right) -f\left( x_{0}\right) \right\vert <\varepsilon
\right\} $ is $\left( p\left( .\right) ,\vartheta \right) $-finely open.
This follows that $f:U\longrightarrow 
\mathbb{R}
$ is continuous at $x_{0}\in U$ in sense to the $\left( p\left( .\right)
,\vartheta \right) $-fine topology on $U$. The converse argument is still an
open problem, see \cite{HarLat}. Moreover, this argument for the constant
exponent was considered by \cite[Teorem 2.136]{Mal}.
\end{remark}

It is note that a set $A$ is a $\left( p\left( .\right) ,\vartheta \right) $%
-fine neighbourhood of a point $x$ if and only if $x\in A$ and $%
\mathbb{R}
^{d}-A$ \ is $\left( p\left( .\right) ,\vartheta \right) $-thin at $x,$ see 
\cite{Hei}.

\begin{theorem}
Let $A\subset B\left( x_{0},r\right) ,$ $cap_{p\left( .\right) ,\vartheta
}\left( A,B\left( x_{0},4r\right) \right) \geq 1$ and $0<r\leq s\leq 2r.$
Assume that $A\subset 
\mathbb{R}
^{d}$ is $\left( p\left( .\right) ,\vartheta \right) $-thin at $x\in 
\overline{A}-A$. Then there exists an open neighbourhood $U$ of $A$ such
that $U$ is $\left( p\left( .\right) ,\vartheta \right) $-thin at $x$ and $%
x\notin U.$
\end{theorem}

\begin{proof}
Using the same methods in the Theorem \ref{fineatif2} and Theorem \ref%
{fineatif1}, it can be found for $r\leq s\leq 2r$ that%
\begin{equation}
cap_{p\left( .\right) ,\vartheta }\left( A\cap B\left( x,r\right) ,B\left(
x,2r\right) \right) \approx cap_{p\left( .\right) ,\vartheta }\left( A\cap
B\left( x,r\right) ,B\left( x,2s\right) \right)  \label{neigh1}
\end{equation}%
and%
\begin{equation}
cap_{p\left( .\right) ,\vartheta }\left( B\left( x,r\right) ,B\left(
x,2r\right) \right) \approx cap_{p\left( .\right) ,\vartheta }\left( B\left(
x,s\right) ,B\left( x,2s\right) \right)  \label{neigh2}
\end{equation}%
where the constants in $\approx $ depend on $r$, $p^{-}$, $p^{+}$, constants
of doubling measure and Poincar\'{e} inequality, see \cite{Unal}. If we
consider the Theorem \ref{fineatif2} and the monotonicity of relative $%
\left( p\left( .\right) ,\vartheta \right) $-capacity, then we have%
\begin{eqnarray*}
cap_{p\left( .\right) ,\vartheta }\left( A\cap \overline{B\left(
x,2^{-i}\right) },B\left( x,2^{1-i}\right) \right) &\leq &Tcap_{p\left(
.\right) ,\vartheta }\left( A\cap \overline{B\left( x,2^{-i}\right) }%
,B\left( x,2^{2-i}\right) \right) \\
&\leq &Tcap_{p\left( .\right) ,\vartheta }\left( A\cap B\left(
x,2^{1-i}\right) ,B\left( x,2^{2-i}\right) \right) .
\end{eqnarray*}%
By the definition of relative $\left( p\left( .\right) ,\vartheta \right) $%
-capacity, it can be taken open sets $U_{i}\supset A\cap \overline{B\left(
x,2^{-i}\right) }$ such that%
\begin{eqnarray}
&&\left( \frac{cap_{p\left( .\right) ,\vartheta }\left( U_{i},B\left(
x,2^{1-i}\right) \right) }{cap_{p\left( .\right) ,\vartheta }\left( B\left(
x,2^{-i}\right) ,B\left( x,2^{1-i}\right) \right) }\right) ^{\frac{1}{%
p\left( x\right) -1}}  \notag \\
&\leq &\left( \frac{cap_{p\left( .\right) ,\vartheta }\left( A\cap \overline{%
B\left( x,2^{-i}\right) },B\left( x,2^{1-i}\right) \right) }{cap_{p\left(
.\right) ,\vartheta }\left( B\left( x,2^{-i}\right) ,B\left(
x,2^{1-i}\right) \right) }\right) ^{\frac{1}{p\left( x\right) -1}}+\frac{1}{%
2^{i}}.  \label{neigh3}
\end{eqnarray}%
Denote%
\begin{equation*}
U=\left( 
\mathbb{R}
^{d}-\overline{B_{1}}\right) \cup \left( U_{1}-\overline{B_{2}}\right) \cup
\left( \left( U_{1}\cap U_{2}\right) -\overline{B_{3}}\right) \cup \left(
\left( U_{1}\cap U_{2}\cap U_{3}\right) -\overline{B_{4}}\right) \cup ...
\end{equation*}%
where $B_{i}=B\left( x,2^{-i}\right) .$ It is easy to see that $U$ is open, $%
A\subset U$ holds and $x\notin U.$ Since $B\left( x,2^{-i}\right) \subset 
\overline{B\left( x,2^{-i}\right) }$ holds for $i\in 
\mathbb{N}
,$ it is clear that $U\cap B\left( x,2^{-i}\right) \subset U_{i}$. By (\ref%
{neigh3}), we have%
\begin{eqnarray*}
&&\left( \frac{cap_{p\left( .\right) ,\vartheta }\left( U\cap B\left(
x,2^{-i}\right) ,B\left( x,2^{1-i}\right) \right) }{cap_{p\left( .\right)
,\vartheta }\left( B\left( x,2^{-i}\right) ,B\left( x,2^{1-i}\right) \right) 
}\right) ^{\frac{1}{p\left( x\right) -1}} \\
&\leq &\left( \frac{cap_{p\left( .\right) ,\vartheta }\left( U_{i},B\left(
x,2^{1-i}\right) \right) }{cap_{p\left( .\right) ,\vartheta }\left( B\left(
x,2^{-i}\right) ,B\left( x,2^{1-i}\right) \right) }\right) ^{\frac{1}{%
p\left( x\right) -1}} \\
&\leq &\left( \frac{cap_{p\left( .\right) ,\vartheta }\left( A\cap \overline{%
B\left( x,2^{-i}\right) },B\left( x,2^{1-i}\right) \right) }{cap_{p\left(
.\right) ,\vartheta }\left( B\left( x,2^{-i}\right) ,B\left(
x,2^{1-i}\right) \right) }\right) ^{\frac{1}{p\left( x\right) -1}}+\frac{1}{%
2^{i}}.
\end{eqnarray*}%
Moreover, if we consider (\ref{neigh1}), then we get%
\begin{eqnarray}
&&cap_{p\left( .\right) ,\vartheta }\left( A\cap \overline{B\left(
x,2^{-i}\right) },B\left( x,2^{1-i}\right) \right)  \notag \\
&\leq &C_{1}cap_{p\left( .\right) ,\vartheta }\left( A\cap \overline{B\left(
x,2^{-i}\right) },B\left( x,2^{2-i}\right) \right)  \notag \\
&\leq &C_{2}cap_{p\left( .\right) ,\vartheta }\left( A\cap B\left(
x,2^{1-i}\right) ,B\left( x,2^{2-i}\right) \right) .  \label{neigh5}
\end{eqnarray}

By (\ref{neigh2}), the inequality%
\begin{eqnarray}
&&cap_{p\left( .\right) ,\vartheta }\left( B\left( x,2^{-i}\right) ,B\left(
x,2^{1-i}\right) \right)  \notag \\
&\geq &\frac{1}{C_{2}}cap_{p\left( .\right) ,\vartheta }\left( B\left(
x,2^{1-i}\right) ,B\left( x,2^{2-i}\right) \right)  \label{neigh6}
\end{eqnarray}%
holds. If we combine (\ref{neigh5}) and (\ref{neigh6}), then we have%
\begin{eqnarray*}
&&\dsum\limits_{i=1}^{\infty }\left( \frac{cap_{p\left( .\right) ,\vartheta
}\left( U\cap B\left( x,2^{-i}\right) ,B\left( x,2^{1-i}\right) \right) }{%
cap_{p\left( .\right) ,\vartheta }\left( B\left( x,2^{-i}\right) ,B\left(
x,2^{1-i}\right) \right) }\right) ^{\frac{1}{p\left( x\right) -1}} \\
&\leq &C\dsum\limits_{i=1}^{\infty }\left( \frac{cap_{p\left( .\right)
,\vartheta }\left( A\cap B\left( x,2^{1-i}\right) ,B\left( x,2^{2-i}\right)
\right) }{cap_{p\left( .\right) ,\vartheta }\left( B\left( x,2^{1-i}\right)
,B\left( x,2^{2-i}\right) \right) }\right) ^{\frac{1}{p\left( x\right) -1}%
}+1.
\end{eqnarray*}%
Since $A$ is $\left( p\left( .\right) ,\vartheta \right) $-thin at $x,$ by
considering the definition of Wiener sum $W_{p\left( .\right) ,\vartheta
}^{sum},$ we conclude%
\begin{equation*}
\dsum\limits_{i=1}^{\infty }\left( \frac{cap_{p\left( .\right) ,\vartheta
}\left( U\cap B\left( x,2^{-i}\right) ,B\left( x,2^{1-i}\right) \right) }{%
cap_{p\left( .\right) ,\vartheta }\left( B\left( x,2^{-i}\right) ,B\left(
x,2^{1-i}\right) \right) }\right) ^{\frac{1}{p\left( x\right) -1}}<\infty .
\end{equation*}%
This follows that%
\begin{eqnarray*}
&&W_{p\left( .\right) ,\vartheta }^{sum}\left( U,x\right) \\
&=&\dsum\limits_{i=0}^{\infty }\left( \frac{cap_{p\left( .\right) ,\vartheta
}\left( U\cap B\left( x,2^{-i}\right) ,B\left( x,2^{1-i}\right) \right) }{%
cap_{p\left( .\right) ,\vartheta }\left( B\left( x,2^{-i}\right) ,B\left(
x,2^{1-i}\right) \right) }\right) ^{\frac{1}{p\left( x\right) -1}} \\
&=&\left. \left( \frac{cap_{p\left( .\right) ,\vartheta }\left( U\cap
B\left( x,1\right) ,B\left( x,2\right) \right) }{cap_{p\left( .\right)
,\vartheta }\left( B\left( x,1\right) ,B\left( x,2\right) \right) }\right) ^{%
\frac{1}{p\left( x\right) -1}}\right. \\
&&\left. +\dsum\limits_{i=1}^{\infty }\left( \frac{cap_{p\left( .\right)
,\vartheta }\left( U\cap B\left( x,2^{-i}\right) ,B\left( x,2^{1-i}\right)
\right) }{cap_{p\left( .\right) ,\vartheta }\left( B\left( x,2^{-i}\right)
,B\left( x,2^{1-i}\right) \right) }\right) ^{\frac{1}{p\left( x\right) -1}%
}<\infty .\right.
\end{eqnarray*}%
Hence $U$ is $\left( p\left( .\right) ,\vartheta \right) $-thin at $x.$ Thus
the claim is follows from definition of open neighbourhood.
\end{proof}

Now, we consider the usage area of $\left( p\left( .\right) ,\vartheta
\right) $-fine topology in potential theory. We define $\left( p\left(
.\right) ,\vartheta \right) $-Laplace equation as%
\begin{equation}
-\Delta _{p\left( .\right) ,\vartheta }=-\func{div}\left( \vartheta \left(
x\right) \left\vert \nabla f\right\vert ^{p\left( .\right) -2}\nabla
f\right) =0  \label{solution}
\end{equation}%
for every $f\in W_{0,\vartheta }^{1,p(.)}\left( \Omega \right) .$

\begin{definition}
(\cite{Unal2}) Let $\Omega \subset 
\mathbb{R}
^{d}$ for $d\geq 2,$ be an open set. A function $f\in W_{\vartheta
,loc}^{1,p(.)}(\Omega )$ called a (weak) weighted solution (briefly $\left(
p\left( .\right) ,\vartheta \right) $-solution) of (\ref{solution}) in $%
\Omega $, if%
\begin{equation*}
\dint\limits_{\Omega }\left\vert \nabla f\left( x\right) \right\vert
^{p\left( x\right) -2}\nabla f\left( x\right) \cdot \nabla g\left( x\right)
\vartheta \left( x\right) dx=0
\end{equation*}%
whenever $g\in C_{0}^{\infty }\left( \Omega \right) .$ Moreover, a function $%
f\in W_{\vartheta ,loc}^{1,p(.)}(\Omega )$ is a (weak) weighted
supersolution (briefly $\left( p\left( .\right) ,\vartheta \right) $%
-supersolution) of (\ref{solution}) in $\Omega $, if%
\begin{equation}
\dint\limits_{\Omega }\left\vert \nabla f\left( x\right) \right\vert
^{p\left( x\right) -2}\nabla f\left( x\right) \cdot \nabla g\left( x\right)
\vartheta \left( x\right) dx\geq 0  \label{supersolution}
\end{equation}%
whenever $g\in C_{0}^{\infty }\left( \Omega \right) $ is nonnegative. A
function $f$ is a weighted subsolution in $\Omega $ if $-f$ is a $\left(
p\left( .\right) ,\vartheta \right) $-supersolution in $\Omega $, and a
weighted solution in $\Omega $.
\end{definition}

\begin{definition}
(\cite{HarLat}, \cite{Hei}) A function $f:\Omega \longrightarrow \left(
-\infty ,\infty \right] $ is $\left( p\left( .\right) ,\vartheta \right) $%
-superharmonic in $\Omega $ if

\begin{enumerate}
\item[\textit{(i)}] f is lower semicontinuous,

\item[\textit{(ii)}] f is finite almost everywhere,

\item[\textit{(iii)}] Assume that $D\Subset \Omega $ is an open set. If $g$
is a $\left( p\left( .\right) ,\vartheta \right) $-solution in $D,$ which is
continuous in $\overline{D},$ and satisfies $f\geq g$ on $\partial D,$ then $%
f\geq g$ in $D.$
\end{enumerate}
\end{definition}

Note that every $\left( p\left( .\right) ,\vartheta \right) $-supersolution
in $\Omega $, which satisfies%
\begin{equation*}
f\left( x\right) =\text{ess}\liminf_{y\longrightarrow x}f\left( y\right)
\end{equation*}%
for all $x\in \Omega ,$ is $\left( p\left( .\right) ,\vartheta \right) $%
-superharmonic in $\Omega .$ On the other hand every locally bounded $\left(
p\left( .\right) ,\vartheta \right) $-superharmonic function is a $\left(
p\left( .\right) ,\vartheta \right) $-supersolution. The proof can be easily
seen by using the similar method in \cite{HarMar}, \cite{Hei}.

Let $S\left( 
\mathbb{R}
^{d}\right) $ be the class of all $\left( p\left( .\right) ,\vartheta
\right) $-superharmonic functions in $%
\mathbb{R}
^{d}$. Since $\left( p\left( .\right) ,\vartheta \right) $-superharmonic
functions are lower semicontinuous and since $S\left( 
\mathbb{R}
^{d}\right) $ is closed under truncations, $\left( p\left( .\right)
,\vartheta \right) $-fine topology is the coarsest topology on $%
\mathbb{R}
^{d}$ making all locally bounded $\left( p\left( .\right) ,\vartheta \right) 
$-superharmonic functions continuous, see \cite{Hei}.

\section{Acknowledgment}

We express our thanks to Professor Jana Bj\"{o}rn for kind comments and
helpful suggestions.

\end{document}